\documentclass[11pt]{article}
\usepackage[utf8]{inputenc}
\usepackage[english]{babel}
\usepackage[deletedmarkup=xout]{changes}
\definechangesauthor[name=DANIELE, color=blue]{DAN}
\definechangesauthor[name=ALESSANDRO, color=violet]{DAN-ALE OK}
\definechangesauthor[name=ALESSANDRO2, color=red]{ALE-GIU}
\usepackage{cancel}
\usepackage{csquotes}
\usepackage{graphicx}
\usepackage{yhmath}
\usepackage{color}
\usepackage[utf8]{inputenc}
\usepackage{appendix}
\usepackage{amsmath}
\usepackage{amsthm}
\usepackage{amsfonts}
\usepackage{amssymb}
\usepackage{longtable}
\usepackage{cite}
\usepackage{url}
\usepackage{graphics}
\usepackage{subcaption}
\usepackage{epsfig}
\usepackage{geometry}
\usepackage{verbatim}
\usepackage{tikz}
\usetikzlibrary{fit,shapes.misc}

\usepackage{listings,xcolor}
\definecolor{dkgreen}{rgb}{0,.6,0}
\definecolor{dkblue}{rgb}{0,0,.6}
\definecolor{dkyellow}{cmyk}{0,0,.8,.3}

\newtheorem{theorem}{Theorem}[section]

\newtheorem{remark}[theorem]{Remark}
\newtheorem{cor}[theorem]{Corollary}
\newtheorem{prop}[theorem]{Proposition}
\newtheorem{definition}[theorem]{Definition}

\def\field{\mathbb{}}

\title{A new infinite family of maximum $h$-scattered $\mathbb{F}_q$-subspaces of $V(m(h+1),q^n)$ and associated MRD codes}

\author{Daniele Bartoli\textsuperscript{1}, Alessandro Giannoni\textsuperscript{2} and Giuseppe Marino\textsuperscript{2}}

\newcommand{\Addresses}{{
		\bigskip
		\footnotesize
		
		\noindent \textsuperscript{1}\textsc{Department of Mathematics and Informatics,\\ 
        University of Perugia,\\ Perugia,\\ Italy}\par\nopagebreak
		\textit{e-mail addresses}: \texttt{daniele.bartoli@unipg.it}\\
		\medskip
  
		\noindent \textsuperscript{2}\textsc{Department of Mathematics and applications "Renato Caccioppoli",\\
			University Federico II of Napoli,\\ Napoli,\\
			Italy}\par\nopagebreak
		\textit{e-mail address}: \texttt{\{alessandro.giannoni, giuseppe.marino\}@unina.it}}}
  
  \date{}

\DeclareMathOperator{\dd}{d}
\DeclareMathOperator{\rk}{rk}

\begin{document}
\maketitle

\begin{abstract}
   The exploration of linear subspaces, particularly scattered subspaces, has garnered considerable attention across diverse mathematical disciplines in recent years, notably within finite geometries and coding theory. Scattered subspaces play a pivotal role in analyzing various geometric structures such as blocking sets, two-intersection sets, complete arcs, caps in affine and projective spaces over finite fields and rank metric codes. This paper introduces a new infinite family of $h$-subspaces, along with their associated MRD codes. Additionally, it addresses the task of determining the generalized weights of these codes. Notably, we demonstrate that these MRD codes exhibit some larger generalized weights compared to those previously identified.
\end{abstract}

\makeatother
\newcommand{\Prf}{\noindent{\bf Proof}.\quad }
\renewcommand{\labelenumi}{(\alph{enumi})}

\def\as{{^{\sigma}}}
\def\e{{\`{e }}}
\def\a{{\`{a}}}
\def\qll{{^{q^{\ell}}}}
\def\qi{{^{q^I}}}
\def\qih{^{q^h}}
\def\qjh{^{q^{h+1}}}
\def\qil{^{q^{\ell+1}}}
\def\qjl{^{q^{\ell+2}}}
\def\qii{{^{q^{2I}}}}
\def\qj{{^{q^J}}}
\def\qjj{{^{q^{2J}}}}
\def\qio{{^{q^{I_0}}}}
\def\qjo{{^{q^{J_0}}}}
\def\qk{{^{q^K}}}
\def\qji{{^{q^{J-I}}}}
\def\qij{{^{q^{I+J}}}}
\def\qjia{{^{q^{J-I}+1}}}
\def\qija{{^{q^{I-J}+1}}}
\def\field{{\mathbb F_{q^n}}}
\def\us{U_{\alpha,\beta}}
\def\fax{{x^{q^I} +\alpha y^{q^J}}}
\def\fbx{{x\qj +\beta z\qi}}
\def\fcx{{y\qi +\gamma z\qj}}
\def\faxl{{\lambda\qi x\qi +\alpha\lambda\qj y\qj}}
\def\fbxl{{\lambda\qj x\qj +\beta\lambda\qi z\qi}}
\def\fcxl{{\lambda\qi y\qi +\gamma\lambda\qj z\qj}}
\newcommand{\fal}[2]{{\lambda\qi #1\qi +\alpha\lambda\qj #2\qj}}
\newcommand{\fbl}[2]{{\lambda\qj #1\qj +\beta\lambda\qi #2\qi}}
\newcommand{\fcl}[2]{{\lambda\qi #1\qi +\gamma\lambda\qj #2\qj}}
\newcommand{\fa}[2]{{#1^{q^I} +\alpha #2^{q^J}}}
\newcommand{\fb}[2]{{#1\qj +\beta #2\qi}}
\newcommand{\fc}[2]{{#1\qi +\gamma #2\qj}}

\def\F{\mathbb{F}}
\def\cD{\mathcal{D}}


\section{Introduction}
Linear subspaces, particularly scattered subspaces, have been extensively studied in recent years across various branches of mathematics, primarily finite geometries and coding theory. These subspaces hold particular interest due to their relevance and applications in different contexts.

In finite geometries, scattered subspaces are essential in the analysis of various structures such as blocking sets, two-intersection sets, complete arcs, caps in affine and projective spaces over finite fields and rank metric codes

Overall, the exploration of scattered subspaces has led to notable advancements in both theoretical understanding and practical implementations, enriching our knowledge of finite geometries and coding theory while also facilitating the development of efficient algorithms and techniques.

This article provides a broader view regarding specific families of scattered subspaces, offering a family of examples that exhibit notable additional characteristics. These constructions showcase the power of polynomial construction in obtaining such objects.


Consider a prime power \( q \) and \( r, n\in\mathbb{N} \). Let \( V \) denote an \( r \)-dimensional vector space over \( \mathbb{F}_{q^n} \).  For any $k$-dimensional $\mathbb{F}_q$-vector subspace $U$ of $V$, the set $L(U) \subset\Lambda=\mathrm{PG}(V, q^n)$ defined by the non-zero vectors of $U$ is called an $\mathbb{F}_q$-\emph{linear set} of $\Lambda$ of \emph{rank} $k$, i.e.
\[ L(U)=\{\langle {\mathbf u} \rangle_{\mathbb{F}_{q^n}}: {\mathbf u} \in U\setminus \{\mathbf{0} \}  \}.\]

Notably, different vector subspaces can define the same linear set. Thus, we always analyze a linear set along with the \( \mathbb{F}_q \)-vector subspace defining it simultaneously.

Let $\Omega=\mathrm{PG}(W,\mathbb{F}_{q^n})$ be a subspace of $\Lambda$ and let $L(U)$ be an $\mathbb{F}_q$-linear set of $\Lambda$. We say that $\Omega$ has \emph{weight} $i$ in $L(U)$ if $\dim_{\mathbb{F}_q}(W\cap U)=i$. In particular, a point of $\Lambda$ belongs to $L(U)$ if and only if it has weight at least $1$. Also, for any $\mathbb{F}_q$-linear set $L(U)$ of rank $k$, 
\[|L(U)|\leq \frac{q^{k}-1}{q-1}.\]
When the equality holds, i.e.\ all the points of $L(U)$ have weight $1$, we call $L(U)$ a \emph{scattered} linear set and $U$ a \emph{scattered} $\F_q$-subpace of $V$. A scattered $\mathbb{F}_q$-linear set $L(U)$ of highest possible rank is called a \emph{maximum scattered} $\mathbb{F}_q$-\emph{linear set} (and $U$ is said to be a \emph{maximum scattered} $\F_q$-subspace).

Recent years have witnessed an exploration of scattered and maximum scattered subspaces through suitable polynomial representations. This method was initiated in \cite{sheekey_new_2016}. For every \( n \)-dimensional \( \mathbb{F}_q \)-subspace \( U \) of \( \mathbb{F}_{q^n}\times \mathbb{F}_{q^n} \), there exists an appropriate basis of \( \mathbb{F}_{q^n}\times \mathbb{F}_{q^n} \) and an \( \mathbb{F}_q \)-linearized polynomial \( f(x)=\sum A_i x^{q^i} \in \mathbb{F}_{q^n}[x] \) with degree less than \( q^n \), such that \( U= \{ (x, f(x)) : x\in \mathbb{F}_{q^n}  \}.\) Subsequently, maximum scattered linear sets in \( \mathrm{PG}(1,q^n) \) can be described via so-called scattered polynomials; refer to \cite{sheekey_new_2016}. Generalizations of these connections led to the index description of scattered polynomials \cite{BZ2018} and scattered sequences \cite{BMNV2022, BGM2024}. 

In \cite{CsMPZ2019}, a particular category of scattered subspaces was introduced.

\begin{definition}
	Let $V$ be an $r$-dimensional $\F_{q^n}$-vector space.
	An $\F_q$-subspace $U$ of $V$ is called $h$-scattered, $0<h \leq r-1$, if $\langle U \rangle_{\F_{q^n}}=V$ and each $h$-dimensional $\F_{q^n}$-subspace of $V$ meets $U$ in an $\F_q$-subspace of dimension at most $h$. An $h$-scattered subspace of highest possible dimension is called a maximum $h$-scattered subspace.
\end{definition}

With this definition, the $1$-scattered subspaces are the scattered subspaces generating $V$ over $\F_{q^n}$. When $h=r$, the above definition holds true for $n$-dimensional $\F_q$-subspaces of $V$, which define subgeometries of $\mathrm{PG}(V,\F_{q^n})$.
If $h=r-1$ and $\dim_{\F_q} U=n$, then $U$ defines a scattered $\F_q$-linear set with respect to hyperplanes; see \cite[Definition 14]{ShVdV}.

In \cite[Theorem 2.3]{CsMPZ2019} it has been proved that for an $h$-scattered subspace $U$ of $V(r,q^n)$, if $U$ does not define a subgeometry, then
\begin{equation}\label{hscatbound}
\dim_{\F_q} U \leq \frac{rn}{h+1}.
\end{equation}
The $h$-scattered subspaces whose dimension reaches Bound \eqref{hscatbound} are called \emph{maximum}. Also, $h$-scattered $\F_q$-subspaces of $V(r,q^n)$ exist whenever $h+1 \mid r$ for any values of $q$ (cf. \cite[Theorem 2.6]{CsMPZ2019}).

Whereas, if $h+1$ is not a divisor of $r$, in \cite[Theorem 3.6]{CsMPZ2019} it has been proved the  existence of maximum $(n-3)$-scattered $\F_q$-subspaces of $V(r(n-2)/2,q^n)$, for $n\geq 4$ even and $r\geq 3$ odd.

In addition, in \cite{BGGM} the authors prove the existence of maximum $2$-scattered $\F_q$-subspaces of $V(r,q^3)$, where $r\geq 3$ and $r\ne 5$, for any value of $q$ and of $V(r,q^6)$, where $r\geq 3$ and $r\ne 5$, with $q=2^h$ and $h\geq 1$ odd; see \cite[Main Theorem]{BGGM}.

In \cite[Corollary 4.4]{ShVdV}, the \( (r-1) \)-scattered subspaces of \( V(r,q^n) \) reaching Bound \eqref{hscatbound}, i.e., of dimension \( n \), have been shown to be equivalent to MRD-codes of \( \F_{q}^{n\times n} \) with a minimum rank distance \( n-r+1 \) and with the left or right idealizer isomorphic to \( \F_{q^n} \). In \cite[Sec. 5]{zini2021scattered}, a connection between maximum \( h \)-scattered subspaces of $V(r,q^n)$, $n\geq h+3$, and MRD $[rn/(h+1),r]_{q^n/q}$ codes with right idealizer isomorphic to $\F_{q^n}$ has been established.

The primary open problem regarding maximum \( h \)-scattered subspaces in \( V(r,q^n) \) is their existence for every admissible value of \( r \), \( n \), and \( h\geq 2 \). It is now known that when \( 2 \mid rn \), scattered subspaces of maximum dimension always exist \cite{BBL2000, BGMP2015, BL2000, CSMPZ2016}.

In this paper, we provide a new infinite family of maximum $h$-scattered $\F_q$-subspaces in
$V(m(h+1),q^n)$, for any integers $h,m,n$, with $2\leq h\leq n-3$ and for any prime power $q$.

Another crucial achievement of this work is to establish limitations on some generalized weights of the codes derived from the new family of $h$-scattered subspaces. These parameters are in general better, hence providing genuinely new MRD codes.

\section{Main Result}\label{sec:construction}
In this section we present a family of maximum $h$-scattered $\F_q$-subspaces of the vector space $\F_{q^n}^{m(h+1)}$.
\bigskip

Let $n,m,h$ be positive integers with $m\geq3$, $2\leq h\leq n-2$
and $q$ be a prime power. For each choice of  $\alpha_1 ,\dots, \alpha_m\in{\mathbb {F}}_{{q^n}}^*$, let $\textbf{A}:=(\alpha_1,\dots,\alpha_m)$ and $K_{\textbf{A}}:=\alpha_1^{q^{(m-1)}}\alpha_2 \alpha_3^{q}\cdots \alpha_m^{q^{(m-2)}}$. Define the following $\F_q$-subspace of $\F_{q^n}^{m(h+1)}$:
\begin{equation}\label{set}
V_{\mathbf{A},h}:=\{(\underline{x},\underline{x}^q,\underline{x}^{q^2},\dots,\underline{x}^{q^{h-1}},f_1(\underline{x}),f_2(\underline{x}),\dots,f_{m-1}(\underline{x}),f_m(\underline{x})):\underline{x}\in\mathbb{F}^m_{q^n}\},
\end{equation}
where $f_m(\underline{x}):=x_m^{q^h}+\alpha_1x_1^{q^{h+1}}$ and $f_i(\underline{x}):=x_i^{q^h}+\alpha_{i+1}x_{i+1}^{q^{h+1}}$, with $i=1,\dots,m-1$. 


In what follows we will make use several times of the following notations. For a given $\mathbb{F}_{q^n}$-subspace $H$ of $V(m(h+1),q^n)$ of dimension $t$, consider the set of  $m(h+1)-t$ independent linear homogeneous equation in $z_1,\ldots,z_{m(h+1)}$. Let $\mathcal{\overline{S}}_{i}$, $i=1,\ldots,h+1$, be the set of such equations containing $z_1,\ldots, z_{im}$ and let  $ \overline{s}_i=|\mathcal{\overline{S}}_{i}|$. Denote by $\mathcal{{S}}_{1} = \mathcal{\overline{S}}_{1}$,  $\mathcal{{S}}_{i} = \mathcal{\overline{S}}_{i}\setminus \mathcal{\overline{S}}_{i-1}$, $i=2,\ldots,h+1$, and  $s_i:= |\mathcal{S}_i|$, $i=1,\ldots,h+1$.     By our assumption on the independence of the $t$ equations defining $H$, the matrix of the corresponding system $(S)$ can be rearranged as follows 
\begin{equation}\label{Eq:matrix}
    \begin{pmatrix}
    A_0^{(h+1)}&A_1^{(h+1)}&\cdots&A_{h-1}^{(h+1)}&A_{h}^{(h+1)}\\
    A_0^{(h)}&A_1^{(h)}&\cdots&A_{h-1}^{(h)}&0\\
    \vdots\\
    A_{0}^{(2)}&A_{1}^{(2)}&\cdots&0&0\\
    A_{0}^{(1)}&0&\cdots&0&0
\end{pmatrix},
\end{equation}
where $A_{i}^{(j)}$, $i=0,\dots,h$, $j=1,\dots,h+1$, is a matrix with $s_j$ rows and $m$ columns and each $A_{i}^{(i+1)}$, $i=0,\dots,h$, is of full rank. 

Consider now the substitution 
$$(z_1,\dots,z_{(h+1)m})\mapsto (\underline{x},\underline{x}^q,\underline{x}^{q^2},\dots,\underline{x}^{q^{h-1}},f_1(\underline{x}),f_2(\underline{x}),\dots,f_{m-1}(\underline{x}),f_m(\underline{x}))$$
and the corresponding system $(S^{\prime})$ in the variables $x_1,\ldots,x_m$. For sake of convenience, in order to bound the number of solutions of the system $(S^{\prime})$ we will consider, in some cases, such a system as a linear system in the variables $x_{i,j}=x_i^{q^j}$, $i=1,\ldots,m$, $j=0,\ldots,n-1$.

\bigskip

Now, we prove our main result.
\begin{theorem}\label{Th:2.1}
Consider the set $$\mathcal{A}:=\{\textbf{A}=(\alpha_1,\dots,\alpha_m): K_{\textbf{A}}\text{ is not a } (1+q+q^2+\cdots+q^{m-1})\text{-power}  \}.$$
 The subspace   $V_{\mathbf{A},h}$ defined in \eqref{set} is a maximum $h$-scattered of $\F_{q^n}^{m(h+1)}$ for every $\textbf{A}\in\mathcal{A}$.
\end{theorem}

\begin{proof}
    Consider an $\mathbb{F}_{q^n}$-subspace $H$ of $\F_{q^n}^{m(h+1)}$ of dimension $h$ and thus defined by $m(h+1)-h$ independent linear homogeneous equation in $z_1,\ldots,z_{m(h+1)}$. 
    
    


Using the same notations as in \eqref{Eq:matrix}, our aim is to bound the number of solutions of $(S^{\prime})$ by $q^h$.


Note that each $s_i$ is at most $m$. Since $s_1+\cdots+s_{h+1}=(h+1)m-h$,  there exists at least $k\in \{1,\ldots, h+1\}$ such that $s_{k}= m$.

We distinguish two cases.

\begin{enumerate}
\item $s_{k}=m$ for some $k \in \{1,\ldots,h\}$.

Note that $k=1$ yields  $(x_1,\ldots,x_m)=(0,\ldots,0)$ and the claim follows.

Consider now $k>1$. Let us consider $x_{i,j}=x_i^{q^j}$, $i=1,\ldots,m$, $j=0,\ldots,n-1$, be independent variables. Each equation of System $(S^\prime)$ can be seen as a linear equation in the variables $x_{i,j}$'s. Let $\mathcal{S}_k=\{F_1,\ldots,F_m\}$ be $s_k=m$ independent equations.  
 The equations $(F_1)^{q^{\ell}},\ldots,(F_m)^{q^\ell}$, $\ell=0,\ldots,n-k$, together with the  $s_1+\cdots+s_{k-1}$ equations are independent. They form a system of at least 
 $$(s_1+\cdots+s_{k-1}) +(n-k+1)m\geq (h+1)m-h-(h+2-k)m +(n-k+1)m=nm-h$$
 independent equations and thus the number of solutions is bounded by $q^h$ solutions.
 
\item $s_{h+1}=m$ and $s_1=\cdots=s_{h}=m-1$. 
Since $s_1=m-1$, without loss of generality we can suppose that for each $i\geq 2$, $x_i=\mu_i x_1$, for some $\mu_i \in \mathbb{F}_{q^n}$. Let $A=(\alpha_2 \mu_2^{q^{h+1}},\alpha_3 \mu_3^{q^{h+1}},\ldots,\alpha_m \mu_m^{q^{h+1}},\alpha_1)$ and $B=(1,\mu_2^{q^h},\ldots,\mu_{m-1}^{q^h},\mu_{m}^{q^h})$ and consider  the $m$ equations in $\mathcal{S}_{h+1}$. They read
$$
\begin{cases}
((a_{1,1},\ldots,a_{1,m})\cdot A ) x_1^{q^{h+1}}+((a_{1,1},\ldots,a_{1,m})\cdot B ) x_1^{q^{h}}+\cdots=0\\
\vdots\\
((a_{m,1},\ldots,a_{m,m})\cdot A ) x_1^{q^{{h+1}}}+((a_{m,1},\ldots,a_{m,m})\cdot B ) x_1^{q^{h}}+\cdots=0,
\end{cases}
$$
where $(a_{i,j})_{i,j}$ is the nonsingular $m\times m$ matrix $A_{h}^{(h+1)}$. First, note that not all these equations are vanishing, since the coefficients of $x_1^{q^{h}}$ cannot vanish simultaneously, given the nonsingularity of $A_{h}^{(h+1)}$ and $B\neq (0,0,\ldots,0)$. 


We show now that $A$ and $B$ are linearly independent. Suppose by way of contradiction that $A=uB$ for some $u \in \mathbb{F}_{q^n}$, $u \neq 0$.
In particular 
$$
\begin{cases}
    \alpha_2\mu_2^{q^{h+1}}=u,\\ \alpha_3\mu_3^{q^{h+1}}=u \mu_{2}^{q^h},\\
    \vdots\\
    \alpha_{m}\mu_m^{q^{h+1}}=u \mu_{m-1}^{q^h},\\
    \alpha_1=u \mu_{m}^{q^h}
\end{cases}$$
and thus 
$$
\begin{cases}
    \alpha_2\mu_2^{q^{h+1}}=u,\\ \alpha_3^{q}\mu_3^{q^{h+2}}=u^{q} \mu_{2}^{q^{h+1}},\\
    \vdots\\
    \alpha_{m}^{q^{(m-2)}}\mu_m^{q^{h+m-1}}=u^{q^{(m-2)}} \mu_{m-1}^{q^{h+m-2}},\\
    \alpha_1^{q^{m-1}}=u^{q^{m-1}} \mu_{m}^{q^{h+m-1}}.
\end{cases}$$
This yields
$$\alpha_1^{q^{(m-1)}}=u^{q^{(m-1)}}\cdots u^{q}u/(\alpha_2 \alpha_3^{q}\cdots \alpha_m^{q^{(m-2)}}),$$

a contradiction to $K_\textbf{A}$ not being a $(1+q+\cdots q^{m-1})$-power in $\mathbb{F}_{q^n}$.

Thus, $A$ and $B$ are linearly independent. 

Suppose without loss of generality that the first equation is not vanishing and let $\lambda_i$ be such that the $i$-th equation is $\lambda_i$ times the first one. This means that 
\begin{equation}\label{Eq1}
\lambda_i(a_{1,1},\ldots,a_{1,m})\cdot A =(a_{i,1},\ldots,a_{i,m})\cdot A \end{equation}
and 
\begin{equation}\label{Eq2}
\lambda_i(a_{1,1},\ldots,a_{1,m})\cdot B =(a_{i,1},\ldots,a_{i,m})\cdot B.\end{equation}

Now Equations \eqref{Eq1} and \eqref{Eq2} are equivalent to $\lambda_i(a_{1,1},\ldots,a_{1,m})-(a_{i,1},\ldots,a_{i,m})$ belonging to $$Z:= \{(z_1,\ldots,z_m) : B\cdot (z_1,\ldots,z_m) =A\cdot (z_1,\ldots,z_m)=(0,0,\ldots,0)\}.$$ Since $A$ and $B$ are linearly independent, $Z$ has dimension $m-2$.  Thus 
$$(a_{i,1},\ldots,a_{i,m})\in Z \oplus \langle(a_{1,1},\ldots,a_{1,m})\rangle,$$
a contradiction to $A_{h}^{(h+1)}$ being non-singular.


\end{enumerate}
\end{proof}

\section{The equivalence issue}

In this section we employ the same notations established in Section \ref{sec:construction}.  
As noted in the introduction, there exist maximum $h$-scattered $\F_q$-subspace in $V(m(h+1),q^n)$ that are direct sum of maximum $h$-scattered $\F_q$-subaspaces in $V(h,q^n)$; see  \cite[Example 2.5 and Theorem 2.6]{CsMPZ2019}. Revisiting the concepts of the evasiveness property and the indecomposability of an $\F_q$-subspace is necessary to prove that  the family introduced in the previous section is new. 

\begin{definition}\rm{(\cite{BaCsMT2020})}
Let $h,r,k,n$ be positive integers such that $h<k$ and $h \le r$. An $\mathbb{F}_q$-subspace $U\subseteq V(k,q^n)$ is said to be \emph{$(h,r)$-evasive} if for every $h$-dimensional $\field$-subspace  $H\subseteq V(k,q^n)$, it holds $\dim_{\mathbb{F}_q}(U\cap H)\leq r$. When $h=r$, an $(h,h)$-evasive subspace is $h$-scattered. Furthermore, when $h=1$,  a $1$-scattered subspace is a scattered subspace.
\end{definition}

\begin{definition}\rm{(cf. \cite[Def. 3.3]{BMNV2022})}
 An $\mathbb{F}_q$-subspace $U$ of $V(k,q^n)$   is said to be \emph{decomposable} if it can be written as a direct sum of two or more $\F_q$-subspaces  which are called \emph{factors} of $U$. Furthermore, 
 $U$ is said to be \emph{indecomposable} if it is not decomposable.
\end{definition}

\begin{theorem}\label{hug1}
     The subspace $V_{\mathbf{A},1}$ is $(2,4)$-evasive for every $\textbf{A}\in\mathcal{A}$.
\end{theorem}
\begin{proof}
     Fix $x_1,\dots,x_m,y_1,\dots,y_m\in\field$ such that the first two rows of the following matrix are independent
  $$M:=\begin{pmatrix}
        x_1&\cdots&x_m&x_1^{q}+\alpha_2x_2^{q^2}&\cdots&x_m^{q}+\alpha_1x_1^{q^2}\\
        y_1&\cdots&y_m&y_1^{q}+\alpha_2y_2^{q^2}&\cdots&y_m^{q}+\alpha_1y_1^{q^2}\\

        z_1&\cdots&z_m&z_1^{q}+\alpha_2z_2^{q^2}&\cdots&z_m^{q}+\alpha_1z_1^{q^2}
    \end{pmatrix}.
    $$
We want to show that this matrix has rank two for a maximum of $q^{4}$ $m$-tuples $\underline{z}$.
Suppose by way of contradiction that $M$ has rank $2$. We divide the proof in two cases.

\begin{enumerate}
    \item There exists $i<j\in \{1,\ldots,m\}$ such that $K:=x_iy_j-x_{j}y_i\neq 0$. 
    
    Without loss of generality, we can suppose $i=1$.

   We consider all the submatrices formed by the columns $C_1,C_j,C_\ell$ with $\ell=1,\dots,2m$.
 Since $M$ is of rank 2, all the determinants $K_\ell$ of the $3\times 3$ matrices formed by $C_1,C_j,C_\ell$, $\ell=1,\ldots,2m$, vanish.  From $K_\ell=0$ with $\ell=2,\dots,m$, $\ell\ne j$, we have $z_\ell=g_\ell(z_1,z_j)$, $\ell=2,\dots,m$, $\ell\neq j$. Obviously, we can consider $g_1(z_1,z_j)=z_1,g_j(z_1,z_j)=z_j$. Also, $\deg(g_\ell)=1$. 
 
{Considering $\ell=m+j-1$ and $\ell=2m$ we get }
\begin{equation*}
    \begin{cases}
        K(g_{j-1}(z_1,z_j)^{q}+\alpha_jz^{q^2}_j)+A_1z_1+A_2z_j=0\\
        K(g_m(z_1,z_j)^q+\alpha_1z_1^{q^2})+B_1z_1+B_2z_j=0,
    \end{cases}
\end{equation*}
where $g_m(z_1,z_j)=D_1z_1+D_2z_j$, $g_{j-1}(z_1,z_2)=E_1z_1+E_2z_j$ . 

 Notice that {the two equations above define two plane curves }
\begin{eqnarray*}
\chi_1&:&K(E^q_1z^q_1+E^q_2z^q_j+\alpha_jz^{q^2}_j)+A_1z_1+A_2z_j=0, \\
 \chi_2&:&K(D_1^qz_1^q+D_2^qz_j^q+\alpha_3z_1^{q^2})+B_1z_1+B_2z_j=0, \end{eqnarray*}

with affine coordinates $(z_1,z_j)$. We can estimate the number of solutions of the previous system by estimating the number of intersections of such curves. Our aim is to use Bézout's theorem, so we need to check the existence of common components. Obviously, if they had a common component, they would also have a common point at infinity. Since \mbox{$\overline{\chi_1}\cap r=\{[1:0:0]\}$} and $\overline{\chi_2}\cap r=\{[0:1:0]\}$, where $r$ is the line at infinity,{ we deduce that this is not the case.} So we can apply Bézout's theorem on $\chi_1$ and $\chi_2${ and bound } the number of solutions of our system {by } $q^{4}$.

\item $\underline{y}=\lambda\underline{x}$.

By symmetry, we can consider $x_1\neq 0$.
We have  $\lambda\notin\mathbb{F}_q$, otherwise the first two rows would not be $\mathbb{F}_q$-independent. Since the third row is a linear combination of the first two, in particular $\underline{z}=\mu\underline{x}$, for some $\mu \in \mathbb{F}_{q^n}$, and thus 
    $$M=\begin{pmatrix}
    x_1&\cdots&x_m&x_1^{q}+\alpha_2x_2^{q^2}&\cdots&x_m^{q}+\alpha_1x_1^{q^2}\\
        \lambda x_1&\cdots&\lambda x_m&\lambda^{q} x_1^{q}+\alpha_2\lambda^{q^2} x_2^{q^2}&\cdots&\lambda^{q} x_m^{q}+\alpha_1\lambda^{q^2} x_1^{q^2}\\
        \mu x_1&\cdots&\mu x_m&\mu^{q} x_1^{q}+\alpha_2\mu^{q^2} x_2^{q^2}&\cdots&\mu^{q} x_m^{q}+\alpha_1\mu^{q^2} x_1^{q^2}
    \end{pmatrix}.$$

Consider the $3\times 3$ submatrices formed by the columns $C_1,C_{m+1}$ and $C_\ell$ with $\ell\in[m+2,\dots,2m]$. 
\begin{eqnarray*}
    &&\mu x_1 \begin{vmatrix}
        x_1^{q}+\alpha_2x_2^{q^2}&x_i^{q}+\alpha_{i+1}x_{i+1}^{q^2}\\
        \lambda^{q} x_1^{q}+\alpha_2\lambda^{q^2} x_2^{q^2}&\lambda^{q} x_i^{q}+\alpha_{i+1}\lambda^{q^2} x_{i+1}^{q^2}
    \end{vmatrix}+A_i\mu^{q}+B_i\mu^{q^2}=0, \text{ with } i=2,\dots,m-1,\\
     &&\mu x_1 \begin{vmatrix}
        x_1^{q}+\alpha_2x_2^{q^2}&x_m^{q}+\alpha_{1}x_{1}^{q^2}\\
        \lambda^{q} x_1^{q}+\alpha_2\lambda^{q^2} x_2^{q^2}&\lambda^{q} x_m^{q}+\alpha_{1}\lambda^{q^2} x_{1}^{q^2}
    \end{vmatrix}+A_m\mu^{q}+B_m\mu^{q^2}=0,
\end{eqnarray*}
for some $A_i,B_i\in \mathbb{F}_{q^n}$, $i=1,\ldots,m$.
The coefficients of $\mu$ are 
\begin{eqnarray*}
   && x_1(\lambda^{q}-\lambda^{q^2})(\alpha_2x_2^{q^2} x_i^{q^2}-\alpha_{i+1}x_{i+1}^{q^2} x_1^{q})\text{,  with } i=2,\dots,m-1,\\
 && x_1(\lambda^{q}-\lambda^{q^2})(\alpha_2x_2^{q^2} x_m^{q^2}-\alpha_{1}x_{1}^{q^2+q} )  .
\end{eqnarray*}

If at least one of the coefficients above is non-vanishing, at least one of the equations in $\mu$ is non-vanishing and thus there are at most $q^2$ possibilities for $\mu$, corresponding to at most $q^2$ distinct vectors $\underline{z}$. Suppose now that they all vanish. We obtain
\begin{align*}
    x_1^{q}\alpha_ix_i^{q^2}-x_{i-1}^{q}\alpha_2x_2^{q^2}=0,&\hspace{8 mm} i=3,\dots,m,\\
    x_1^{q+q^2}\alpha_1-\alpha_2x_m^{q} x_2^{q^2}=0.\phantom{...}&
\end{align*}
Letting $y_i:=x_i^{q}/x_1^{q} $ for $ i=2,\dots,m$ and dividing each equation  by $x_1^{q+q^2}$ we get
\begin{align*}
   (E_i)\hspace{20mm}\,\,\, \alpha_iy_i^q-\alpha_2 y_{i-1}y_2^q=0,\hspace{8 mm} i=3,\dots,m,&\phantom{ffffffffffff}\,\,\,\\
    (E)\hspace{23mm}\alpha_1-\alpha_2y_m y_2^q=0.\phantom{......rrrrrrrrrrrr.rrr}
\end{align*}
From $(E_m)$ we obtain $$y_{m-1}=\frac{\alpha_m}{\alpha_2}\frac{y_m^q}{y_2^q}.$$
Substituting \(y_{m - 1}\) in \((E_{m-1})\), we obtain 
$$y_{m-2}=\frac{\alpha_{m-1}\alpha_m^q}{\alpha_2^{q+1}}\frac{y_m^{q^{2}}}{y_2^{q^{2}+q}}.$$
We can iterate this procedure and get
\begin{equation}\label{eqlunga}
    y_{m-(m-2)}=y_2=\frac{\alpha_3\cdot\alpha_4^q\dots\alpha_m^{q^{(m-3)}}}{\alpha_2^{1+q+\dots+q^{m-3}}}\frac{y_m^{q^{(m-2)}}}{y_2^{q+\dots+q^{m-2}}}.
    \end{equation}
From $(E)$ we obtain $y_m=\alpha_1/(\alpha_2y_2^q)$ and substituting in (\ref{eqlunga}) we get
\begin{eqnarray*}
    y_2&=&\frac{\alpha_3\cdot\alpha_4^q\dots\alpha_m^{q^{(m-3)}}}{\alpha_2^{1+q+\dots+q^{m-3}}}\cdot \frac{\alpha_1^{q^{(m-2)}}}{\alpha_2^{q^{(m-2)}}}\cdot \frac{1}{y_2^{q+\dots+q^{m-2}}\cdot y_2^{q^{(m-1)}}}\\
    &=&\frac{\alpha_3\cdot\alpha_4^q\dots\alpha_m^{q^{(m-3)}}\cdot \alpha_1^{q^{(m-2)}}}{\alpha_2^{1+q+\dots+q^{m-2}}}\cdot \frac{1}{y_2^{q+\dots+q^{m-1}}}.
\end{eqnarray*}
Thus $$y_2^{1+q+\dots+q^{m-1}}=\left(\frac{K_{\textbf{A}}}{\alpha_2^{1+q+\dots+q^{m-1}}}\right)^{q^{-1}},$$ a contradiction to our hypothesis. 
\end{enumerate}
\end{proof}

\begin{theorem}\label{hug2}
The subspace     $V_{\mathbf{A},2}$ is $(3,4)$-evasive for every $\textbf{A}\in\mathcal{A}$.
\end{theorem}
\begin{proof}
    
Consider an $\mathbb{F}_{q^n}$-subspace $H$ of $\F_{q^n}^{3m}$ of dimension $3$ and thus defined by $3(m-1)$ independent linear homogeneous equation in $z_1,\ldots,z_{3m}$. 
    

    
In this case, System \eqref{Eq:matrix} reads
$$
    \begin{pmatrix}
    A_0^{(3)}&A_1^{(3)}&A_{2}^{(3)}\\
    A_{0}^{(2)}&A_{1}^{(2)}&0\\
    A_{0}^{(1)}&0&0
\end{pmatrix},
$$
where each $A_{j}^{(j+1)}$, $j=0,1,2$, is of full rank. 


Our aim is to bound the number of solutions of $(S^{\prime})$ by $q^{4}$.

We distinguish a number of cases.

\begin{enumerate}
\item $s_{k}=m$ with $k \in \{1,2\}$.

Note that $k=1$ yields  $(x_1,\ldots,x_m)=(0,\ldots,0)$ and the claim follows.

Consider now $k=2$. Consider the $s_2=m$ independent equations $\mathcal{S}_2=\{F_1,\ldots,F_m\}$.  
 The equations $(F_1)^{q^{\ell}},\ldots,(F_m)^{q^\ell}$, $\ell=0,\ldots,n-2$, together with the  $s_1$ equations in $\mathcal{S}_1$ are independent in the variables $x_{i,j}$. They form a system of at least 
 $$s_1 +(n-1)m\geq 3(m-1)-2m +(n-1)m=nm-m+3m-3-2m=nm-3$$
 independent equations and thus the number of solutions is bounded by  $q^{3}$ solutions.
 
 \item $s_{3}=m, s_1,s_2\leq m-1$

If $s_1=m-1$, without loss of generality we can suppose that there exist $\mu_i\in \mathbb{F}_{q^n}$, $i=2,\ldots,m$, such that
$$(x_1,\ldots,x_m)=(x_1,\mu_2 x_1,\ldots,\mu_m x_1).$$
Substituting it in the $s_{3}=m$ equations in $\mathcal{S}_3$ we notice that there exists at least one equation with non-vanishing coefficient of $x_1^{q^J}$. In fact, such coefficients are $A_2^{(3)}(1,\mu_2^{q^J} ,\ldots,\mu_m^{q^J})^T$ and this cannot be the $0$-vector, being $A_2^{(3)}$ non-singular. Thus there will be at least one non-zero polynomial in $x_1$, of degree $q^J=q^{3}$. So we have at most $q^{3}$ solutions.

If $s_1=m-2$, then $s_2=m-1$. Let $F_1,\dots,F_m$ be the $s_{3}=m$ independent equations in $\mathcal{S}_3$ and $G_1,\dots,G_{m-1}$  be the $s_2=m-1$ independent equations in $\mathcal{S}_2$. The equations $(F_1)^{q^{\ell}},\ldots,(F_m)^{q^\ell}$, $\ell=0,\ldots,n-4$, together with  $(G_1)^{q^{t}},\ldots,(G_{m-1})^{q^t}$, $t=0,1$ and the $s_1$ equations in $\mathcal{S}_1$ are independent. They form a system of 
 $$m(n-3)+(m-1)2+m-2=nm-3m+2m-2+m-2=mn-4$$
 independent equations and thus the number of solutions of $(S^\prime)$ is bounded by  $q^{4}$.

\item $s_1=s_2=s_{3}=m-1$. 

Considering only the $s_1=m-1$ linear equations in $\mathcal{S}_1$, without loss of generality we can suppose that there exist $\mu_i\in \mathbb{F}_{q^n}$, $i=2,\ldots,m$, such that
$$(x_1,\ldots,x_m)=(x_1,\mu_2 x_1,\ldots,\mu_m x_1).$$

If for $k=2$ or $k=3$ the matrix $A^{(k)}_0$ is not proportional to $A^{(1)}_0$, when considering $x_i=\mu_ix_1$ in the $s_k$ equations belonging to $\mathcal{S}_{k}$, at least one of them is nonvanishing (as the equation in $x_1$) and thus the solutions of $(S^{\prime})$ are at most $q^{3}$. 

We are left we the case $A^{(i)}_0=\gamma_iA^{(1)}_0$ for each $i=2,3$, with $\gamma_i \in \mathbb{F}_{q^n}$. Thus, we can reduce the system and consider without loss of generality $A^{(2)}_0=A^{(3)}_0=0$. 

Considering the equations $s_2+s_3=2(m-1)$ in $\mathcal{S}_2 \cup \mathcal{S}_{3}$ raised to the power $q^{n-1}$, we want to prove that these equations have at most $q^{4}$ solutions. Let $H'$ be the $\mathbb{F}_{q^n}$-subspace of $\F_{q^n}^{2m}$ of dimension $2$ defined by 
$$
    \begin{pmatrix}
    A_1^{(3)}&A_{2}^{(3)}\\
    A_{1}^{(2)}&0
\end{pmatrix}^{q^{n-1}}\begin{pmatrix}
    z_1\\z_2\\\vdots\\z_{2m}
\end{pmatrix}=0,
$$
which now is a subsystem of $(S)$.

The size of $H'\cap V_{\textbf{A}^{q^{n-1}},1}$, which is by Theorem \ref{hug1} at most $q^{4}$, provides the desired bound for the number of solutions of $(S^{\prime})$.
\end{enumerate}
\end{proof}

\begin{theorem}\label{thm:evas}
       The subspace $V_{\mathbf{A},h}$ is $(h+1,h+2)$-evasive for every $\textbf{A}\in\mathcal{A}$ with $h\geq2$.
\end{theorem} \begin{proof}We proceed by induction on $h$.
\begin{enumerate}
    \item[-] $h=2$: 

It is Theorem \ref{hug2}.

\item[-] $h>2$

 Consider an $\mathbb{F}_{q^n}$-subspace $H$ of $\F_{q^n}^{m(h+1)}$ of dimension $h+1$ and thus defined by $(m-1)(h+1)$ independent linear homogeneous equation in $z_1,\ldots,z_{m(h+1)}$.

Using the same notations as in \eqref{Eq:matrix}, our aim is to bound the number of solutions of $(S^{\prime})$ by $q^h$.

We distinguish a number of cases.

\begin{enumerate}
\item $s_{k}=m$ for some $k \in \{1,\ldots,h\}$.

Note that $k=1$ yields  $(x_1,\ldots,x_m)=(0,\ldots,0)$ and the claim follows.

Consider now $k>1$. Let $\mathcal{S}_k=\{F_1,\ldots,F_m\}$.
 The equations $(F_1)^{q^{\ell}},\ldots,(F_m)^{q^\ell}$, $\ell=0,\ldots,n-k$, together with the  $s_1+\cdots+s_{k-1}$ equations are independent. They form a system of at least 
 \begin{eqnarray*}
   (s_1+\cdots+s_{k-1}) +(n-k+1)m&\geq& (h+1)(m-1)-(h+2-k)m\\
   &&+(n-k+1)m=nm-h-1  
 \end{eqnarray*}
 independent equations and thus the number of solutions of $(S^{\prime})$ is bounded by  $q^{h+1}$ solutions.
 
 \item $s_{h+1}=m, s_1,\dots,s_h\leq m-1$
Note that $s_1\geq m-2$.
If $s_1=m-1$, without loss of generality we can suppose that
there exist $\mu_i\in \mathbb{F}_{q^n}$, $i=2,\ldots,m$, such that
$$(x_1,\ldots,x_m)=(x_1,\mu_2 x_1,\ldots,\mu_m x_1).$$

Substituting it in the $s_{h+1}=m$ equations in $\mathcal{S}_{h+1}$, there will be at least one non-zero polynomial, of degree $q^J=q^{h+1}$ in $x_1$. So we have at most $q^{h+1}$ solutions, arguing as Case (b) in the proof of Theorem \ref{hug2}.

If $s_1=m-2$, then $s_2=m-1$. Let $\mathcal{S}_{h+1}=\{ F_1,\dots,F_m\}$ and  $\mathcal{S}_{2}=\{G_1,\dots,G_{m-1}\}$. The equations $$(F_1)^{q^{\ell}},\ldots,(F_m)^{q^\ell},\qquad \ell=0,\ldots,n-h-2,$$
together with  
$$(G_1)^{q^{t}},\ldots,(G_{m-1})^{q^t}, \qquad t=0,\ldots,h-1,$$ and the $s_1$ equations in $\mathcal{S}_1$ are independent. They form a system of 
 $$m(n-h-1)+(m-1)h+m-2=nm-mh-m+hm-h+m-2=mn-h-2$$
 independent equations and thus the number of solutions of $(S^{\prime})$ is bounded by  $q^{h+2}$.

\item $s_1=\cdots=s_{h+1}=m-1$. 

Consider the $s_1=m-1$ linear equations in $\mathcal{S}_{1}$. Without loss of generality we can suppose that 
there exist $\mu_i\in \mathbb{F}_{q^n}$, $i=2,\ldots,m$, such that
$$(x_1,\ldots,x_m)=(x_1,\mu_2 x_1,\ldots,\mu_m x_1).$$

If for some $k=2,\ldots,h+1,$ the matrix $A^{(k)}_0$ is not proportional to $A^{(1)}_0$, when considering $x_i=\mu_ix_1$ in the $s_k$ equations in $\mathcal{S}_k$, at least one of them is non-vanishing (seen as equation in $x_1$) and this provides an upper bound of $q^{h+1}$ solutions to system $(S^{\prime})$. Thus, we are left wit the case $A^{(i)}_0=\gamma_iA^{(1)}_0$, for each $i=2,\ldots,h+1$, with $\gamma_i \in \mathbb{F}_{q^n}$. We can reduce the system and consider without loss of generality    $A^{(2)}_0=\cdots=A^{(h+1)}_0=0$. 

Consider the $s_2+\cdots+s_{h+1}=h(m-1)$ equations in $\mathcal{S}_2\cup \cdots \cup \mathcal{S}_{h+1}$ raised to the power  $q^{n-1}$.

In order to show that  this system of  equations has at most $q^{h+1}$ solutions, we consider the  $\mathbb{F}_{q^n}$-subspace $H'$ of dimension $h$ of $V(mh,q^{n})$ defined by 
$$
    \begin{pmatrix}
    A_1^{(h+1)}&\cdots&A_{h-1}^{(h+1)}&A_{h}^{(h+1)}\\
    A_1^{(h)}&\cdots&A_{h-1}^{(h)}&0\\
    \vdots\\
    A_{1}^{(2)}&\cdots&0&0
\end{pmatrix}^{q^{n-1}}\begin{pmatrix}
    z_1\\z_2\\\vdots\\z_{hm}
\end{pmatrix}=0,
$$

which now is a subsystem of $(S)$.

The size of $H'\cap V_{\textbf{A}^{q^{n-1}},h-1}$, which is by induction at most $q^{h+1}$, provides the desired bound for the number of solutions of $(S^{\prime})$.
\end{enumerate}
\end{enumerate}

\end{proof}

Summing up, we have the following result. 

\begin{theorem}\label{thm:new}
  The family of maximum $h$-scattered $\F_q$-subspaces $V_{\mathbf{A},h}$ of $\F_{q^n}^{m(h+1)}$ is new for every  $2\leq h\leq n-3$ .
\end{theorem}
\begin{proof}
    It is sufficient to note that, since $h<n-2$, taking Theorem \ref{thm:evas} into account, the subspace $V_{\mathbf{A},h}$ is not the direct sum of maximum $h$-scattered $\F_q$-subspaces of $(h+1)$-dimensional $\F_{q^n}$-subspaces. 
\end{proof}

\section{Rank-metric codes}

{\it Rank-metric} codes were originally introduced by Delsarte in the late $70$'s \cite{delsarte1978bilinear}, and then resumed a few years later by Gabidulin in \cite{gabidulin1985theory}.

Let $n,t \in \mathbb{N}$ be two positive integers such that $n,t \geq 2$, and let $q$ be a prime power.  
Let consider $v=(v_1,v_2,...,v_t) \in \F_{q^n}^t$, the \textit{rank weight} of $v$ is defined as
$$\omega_{\mathrm{rk}}(v)=\dim \langle v_1,v_2,\ldots,v_t\rangle_{\F_q}.$$
A rank-metric code ${\cal C} \subseteq \F_{q^n}^t$ of {\it length} $t$, is a subset of $\F_{q^n}^t$ considered as a vector space endowed with the metric defined by the map $$d(v,w)= \omega_{rk}(v-w),$$
where $v$ and $w \in \F_{q^n}^t.$ Elements of ${\cal C}$ are called {\it codewords}. A {\it linear} rank-metric code ${\cal C}$ is an $\F_{q^n}$-subspace of $\F_{q^n}^t$ endowed with
the rank metric. If ${\cal C} \subseteq \F_{q^n}^t$ is a linear rank-metric code, the minimum distance between two distinct codewords of ${\cal C}$ is $$d=d({\cal C})=\min \{\omega_{rk}(v) \, | \, v \in \cal C \setminus \{{\bf 0}\}\}.$$ 

\noindent A linear rank-metric code ${\cal C} \subseteq \F_{q^n}^t$ of length $t$, dimension $k$, and minimum distance $d$ is referred in the literature as to an $[t,k,d]_{q^n/q}$ code, or as to an $[t,k]_{q^n/q}$ code, depending on whether the minimum distance is known or not.  
These parameters are related by an inequality, which is known as the Singleton-like bound. Precisely, if $\mathcal{C}$ is an $[t,k,d]_{q^n/q}$ code, then 
\begin{equation}\label{singleton-bound}
    nk \leq \min\{n(t - d + 1), t(n - d + 1)\}; 
\end{equation}
see \cite{delsarte1978bilinear}.

Codes attaining this bound with equality are called {\it maximum rank distance (MRD) codes}, and they are considered to be optimal, due to their largest possible error-correction capability. From the classification of $\F_{q^n}$-linear isometries (see \cite{Berger}), we say that two $[t,k,d]_{q^n/q}$ codes ${\cal C} _1, {\cal C} _2$ are {\it (linearly) equivalent} if there exist $A\in \mathrm{GL}(t,q)$ and $a\in\F_{q^n}^*$ such that \begin{equation*}
{\cal C} _2=a{\cal C} _1\cdot A=\{avA \,:\, v \in {\cal C}_1\}.
\end{equation*}

The geometric counterpart of non-degenerate rank-metric codes (that is,
the columns of any generator matrix of $\mathcal{C}$ are $\F_q$-linearly independent) are the so-called
$q$-systems. Let $U$ be an $\F_q$-subspace of $\F_{q^n}^k$ and let $H$ be an $\F_{q^n}$-subspace of $\F_{q^n}^k$. The \textit{weight} of $H$ in $U$ is $\mathrm{wt}_{U}(H)=\dim_{\F_q}(H\cap U)$. Assume now that $U$ has dimension $t$ over $\F_q$. 

We say that $U$ is an $[t,k,d]_{q^n/q}$ \textit{system} if $\langle U \rangle_{\F_{q^n}}=\F_{q^n}^k$ and 
\begin{equation*}
d=\,t-\mathrm{max}\{\mathrm{wt}_{U}(H) \,:\, H \subseteq \F_{q^n}^k \textnormal{ with } \dim_{\F_{q^n}}(H)=k-1 \}.
\end{equation*}

    More generally, for each $1 \leq \rho \leq k-1$, the parameters

\begin{equation*}
d_{\rho}=\,t-\mathrm{max}\{\mathrm{wt}_{U}(H) \,:\, H \subseteq \F_{q^n}^k \textnormal{ with } \dim_{\F_{q^n}}(H)=k-\rho \}
\end{equation*}
are known as the {\it $\rho$-generalized rank weight} of the system $U$; see \cite[Definition 4]{Randrianarisoa2020geometric}.

 As before, if the parameter $d$ is not relevant, we will write that $U$ is an $[t,k]_{q^n/q}$ system. Furthermore, when none of the parameters is relevant, we will generically refer to $U$ as to a $q$-system.

 Two $[t,k,d]_{q^n/q}$ systems $\mathcal{U}_1,\mathcal{U}_2$ are (linearly) equivalent if there exists $A\in \mathrm{GL}(k,q^n)$ such that 
 \begin{equation*}
 \mathcal{U}_1\cdot A:=\{uA \,:\, u \in \mathcal{U}_1\}=\mathcal{U}_2.
 \end{equation*}

In \cite{Randrianarisoa2020geometric}, a one-to-one correspondence between  equivalence classes of non-degenerate codes  and equivalence classes of the $q$-systems is established.

Moreover, if $\mathcal{C}$ is the $[t,k,d]_{q^n/q}$ code associated to $\mathcal{U}$, the $\rho$-generalized weight of $\mathcal{C}$ is simply defined as the $\rho$-generalized weight of one associated system $\mathcal{U}$; see \cite[Definition 5]{Randrianarisoa2020geometric}. In this case we say that $\mathcal{U}$ is an $[t,k,(d_1,\ldots,d_k)]_{q^n/q}$ system and $\mathcal{C}$ is the $[t,k,(d_1,\ldots,d_k)]_{q^n/q}$ associated code. It is easy to see that the first generalized weight $d_1$ of $\mathcal{C}$ coincides with its minimum distance $d$. 

By \cite{zini2021scattered}, there is a correspondence between maximum $h$-scattered $\F_q$-subspaces of $V(r,q^n)$, with $n\geq h+3$, i.e. $[rn/(h+1),r,n-h]_{q^n/q}$ systems, and  $\F_q$-linear MRD codes with parameters $[rn/(h+1),r,n-h]_{q^n/q}$ and right idealizer isomorphic to $\F_{q^n}$.
By Theorem \ref{thm:new}, we get the following result.
\begin{theorem}
    The MRD code associated with the set $V_{\mathbf{A},h}$ defined in \eqref{set} is new for every $2\leq h\leq n-3$.
\end{theorem}

In what fallows, we want to compare the generalized rank weights of the MRD codes arising from maximum decomposable $h$-scattered subspace with those of our family.  

If $\mathcal{C}$ is an $[t,k]_{q^n/q}$ code, we define $\mathcal{C}^\perp$ to be  its orthogonal complement with respect to the standard inner product on $V(t,q^n)$. In other words, $\mathcal{C}^\perp$ is the $[t,t-k]_{q^n/q}$ code given by
$$ \mathcal{C}^\perp=\{ u \in V(t,q^n) \,:\, uv^\top=0, \mbox{ for every } v \in \mathcal{C} \}$$
and it is called the \emph{dual code} of $\mathcal{C}$.

\begin{remark}\rm{
In \cite{CsMPZ2019}, the authors introduced a duality relation, referred to as \emph{Delsarte duality}, between maximum $h$-scattered subspaces of $V(r,q^n)$ reaching Bound \eqref{hscatbound} and maximum $(n-h-2)$-scattered subspaces of $V(rn/(h+1)-r,q^n)$ reaching Bound \eqref{hscatbound}. 

By \cite[Result 4.7 and Theorem 4.12]{CsMPZ2019} if $\mathcal{C}$ is the MRD code associated with a maximum $h$-scattered subspace $U$ then, up to equivalence, the MRD code associated to the  Delsarte dual of $U$ is  dual code of $\mathcal{C}$. Since two codes $\mathcal{C}$ and $\mathcal{C}'$ are equivalent if and only the corresponding Delsarte duals are equivalent (see, e.g., \cite{sheekey_new_2016}), also the maximum $(n-h-2)$-scattered $\F_q$-subspaces of $V(m(n-h-1),q^n)$, with $2\leq h\leq n-3$, obtained by the Delsarte dual operation on $V_{\mathbf{A},h}$ subspaces of Theorem \ref{thm:new} (and the related MRD codes) are new.
}
\end{remark}

\begin{prop}[\textnormal{see \cite{kurihara2015relative,ducoat2015generalized}}]\label{prop:gen_weights_properties}
 Let $\mathcal{C}$ be an $[t,k,(d_1,\ldots,d_k)]_{q^n/q}$ code and let $\mathcal{C}^\perp$ be its dual $[t,t-k,(d_1^\perp,\ldots,d_{t-k}^\perp)]_{q^n/q}$ code. Then
 \begin{enumerate}
     \item[$(1)$] $1\leq d_1<d_2<\ldots<d_k\leq t$; \hfill \textbf{(Monotonicity)}
     \item[$(2)$] $\{d_1, \ldots,d_k\}\cup \{t+1-d_1^\perp,\ldots,  t+1-d_{t-k}^\perp\}=\{1,\ldots,t\}$. \hfill  \textbf{(Wei-type duality)} 
 \end{enumerate}
\end{prop}
\bigskip

Now we determine the generalized rank weights of an MRD code arising from a maximum $h$-scattered $\F_q$-subspace of $V(m(h+1),q^n)$ obtained as direct sum of $m$ maximum scattered (with respect to hyperplanes) $\F_q$-subspaces in $V(h+1,q^n)$ spaces. 

\begin{prop}\label{prop:genweights_directsum}
 Let ${\mathcal{C}}_1,\ldots, {\mathcal{C}}_m$ be $[n,h+1,n-h]_{q^n/q}$ MRD codes. Then,  $\mathcal{C}:=\bigoplus_{i=1}^m {\mathcal{C}}_i$ is an $[mn,m(h+1),n-h]_{q^n/q}$ code whose  generalized rank weights are given by  
\begin{eqnarray*}
d_{i}(\mathcal{C})&=& n-h-1+i, \quad   \textrm{if}  \quad 1 \leq i \leq h+1, \\
d_{m(h+1)-i}(\mathcal{C})&=&mn-i, \quad   \textrm{if}  \quad  0 \leq i \leq h,\\
d_{(m-1)(h+1)}(\mathcal{C})&=&(m-1)n.  
\end{eqnarray*}
  Also, if $m>2$,
 $$n+1\leq d_{h+2}<d_{h+3}<\cdots< d_{(m-1)(h+1)-1}\leq (m-1)n-1$$ and, in particular, 
 $$d_{k(h+1)}(C)\leq kn $$ for any $k=1,\dots,m-1$.
\end{prop}

\begin{proof}
Let us divide the proof in five steps. 

\begin{enumerate}
    \item First of all, by the monotonicity of Proposition \ref{prop:gen_weights_properties}(1), since each $\mathcal{C}_j$, with $j\in\{1,\ldots,m\}$, is an MRD code of length $n$, we  have
    $$d_{i}(\mathcal{C}_j)=n-(h+1)+i, \qquad \mbox{ for each } i \in \{1,\ldots,h+1\}.$$

\item  Let us write 
$$d_i:=d_{i}(\mathcal{C}),\quad \mbox{ for each } i \in \{1,\ldots, m(h+1)\}.$$ Since every codeword in $\mathcal{C}$ is of the form $(u_1  \mid  u_2  \mid \cdots \mid u_m)$, with $u_j\in \mathcal{C}_j$ and $j\in\{1,\ldots,m\}$,  it is clear that the minimum rank distance of $\mathcal{C}$ (which is the first generalized rank weight) coincides with the minimum among the minimum rank distances of $\mathcal{C}_j$'s, i.e.
 $$d_1=n-h.$$
 Now, for each $j\in\{1,\ldots,m\}$, let $V_j=V(h+1,q^m)$ and $U_j$ be an $[n,h+1,(n-h,n-h+1,\ldots,n)]_{q^m/q}$ system associated with $\mathcal{C}_j$. The $[mn,m(h+1)]_{q^n/q}$ system $U:=\bigoplus_{j=1}^m U_j \subseteq V:=\bigoplus_{i=j}^mV_j$ is then associated with the code $\mathcal{C}$. Now, for each $i \in \{2,\ldots,h+1\}$, consider an $\F_{q^n}$-subspace of codimension $i$ in $V$ given by $\Pi_{i,t}=\bigoplus_{j=1, j\ne t}^m V_j\oplus H_{i,t}$, where $H_{i,t}$ is an $\F_{q^n}$-subspace of codimension $i$ in $V_t$, for a given $t\in\{2,\ldots,m\}$, such that 
 $$d_{i}(\mathcal{C}_t)=n-\dim_{\F_q}(U_t\cap H_{i,t})=n-(h+1)+i.$$ Then, 
 for each $i \in \{2,\ldots,h+1\}$ we have  
 \begin{align*} d_i &= mn-\max\{\dim_{\F_q}(U\cap \Pi) \,:\, \Pi \subseteq V, \ \dim_{\mathbb{F}_{q^m}}(\Pi)=m(h+1)-i \}  \\
 &\leq mn-\dim_{\F_q}(U\cap \Pi_{i,t})\\
 &=mn-\dim_{\F_q}\left(\bigoplus_{j=1,j\ne t}^mU_j\right)+\dim_{\F_q}(U_t\cap H_{i,t})=n-(h+1)+i. \end{align*}
 On the other hand, since $d_1=n-h$, by the monotonicity of Proposition \ref{prop:gen_weights_properties}(1) we have $d_i\geq n-(h+1)+i$ for each $i \in \{2,\ldots,h+1\}$, and thus equality.
 
 \item Let us now consider the dual code $\mathcal{C}^\perp$, which is an $[mn,m(n-h-1),(d_1^\perp,\ldots, d_{m(n-h-1)}^\perp)]_{q^n/q}$ code. It is easy to see that $\mathcal{C}^\perp=\bigoplus_{j=1}^m\mathcal{C}_j^\perp$. However, since the dual of an MRD code is itself MRD,  by step (a), for each $j\in\{1,\ldots,m\}$, $\mathcal{C}_j^\perp$ is an $[n,n-h-1,(h+2,\ldots,n)]_{q^n/q}$ code, and by step (b) the first $n-h-1$ generalized rank weights of $\mathcal{C}^\perp$ are 
 $$d_i^\perp=n-(n-(h+1))+i=h+1+i, \quad \mbox{ for each } i \in \{1,\ldots, n-h-1\}.$$
 
 \item  Since $U$ is $h$-scattered, $d_{m(h+1)-h}\geq mn-h$. Thus, by the monotonicity of Proposition \ref{prop:gen_weights_properties}(1)
$$ d_{m(h+1)-h+i}\geq mn-h+i, \quad \mbox{ for each } i \in \{0,\ldots,h\},$$
which yields 
 $$d_{m(h+1)-h+i} =mn-h+i, \quad \mbox{ for each } i \in \{0,\ldots,h\}.$$

 Recall that $d_1=n-h$, \ldots, $d_{h+1}=n$. Also, 
 $d_1^\perp=h+2$,\ldots, $d_{n-h-1}^\perp=n.$
 
 By the Wei-type duality of Proposition \ref{prop:gen_weights_properties}(2),  we  have 
 $$ \{d_{h+2},\ldots,d_{(m-1)(h+1)} \}\cup \{mn+1-d_{n-h}^\perp,\ldots,mn+1-d_{m(n-h-1)}^\perp \}=$$ $$=\{1,\ldots, n-h-1\}\cup \{n+1,\ldots,(m-1)n\}. $$
Since $d_{h+1}=n$,
$$\{d_{h+2},\ldots,d_{(m-1)(h+1)}\}\subset \{n+1,\ldots,(m-1)n\}.$$

 \item Since the $\F_q$-subspace $U$ associated with $\mathcal{C}$ is $(h+1,n)$-evasive and not $(h+1,n-1)$ evasive, from \cite[Thm. 3.3]{marino2023evasive} $d_{(m-1)(h+1)}=n(m-1)$, the claim follows.

 \end{enumerate}
We have $d_{k(h+1)}(C)\leq kn$ for any $k=1,\dots,m-1$ since the subspace $H:=\bigoplus_{i=1}^{m-k} {\mathcal{C}}_i$ has dimension $(m-k)(h+1)$ and $\text{wt}(H)=(m-k)n$.
 
\end{proof}

\begin{remark}
    \rm{
If $\mathcal{C}$ is an $[mn,m(h+1),n-h]_{q^n/q}$ MRD code arising from a maximum $h$-scattered $\F_q$-subspace $U$ of $V(m(h+1),q^n)$, from \cite[Thm. 3.3]{marino2023evasive} we get $d(\mathcal{C})\geq mn-h$. Since, due to its size, $U$ cannot be $(h+1)$-scattered, we have that $mn-h-2$ is the largest possible value for the $(m-1)(h+1)$-th generalized rank weight of a code with these parameters.    
    }
\end{remark}

\begin{theorem}\label{genweis1}
      The subspace $V_{\mathbf{A},h}$ is $(m(h+1)-(h+1),mn-(2n-2h-2))$-evasive for every $\textbf{A}\in\mathcal{A}$ with $h\geq2$.
\end{theorem}
\begin{proof}
     Consider an $\mathbb{F}_{q^n}$-subspace $H$ of $\F_{q^n}^{m(h+1)}$ of dimension $(m-1)(h+1)$ and thus defined by $(h+1)$ independent linear homogeneous equation in $z_1,\ldots,z_{m(h+1)}$. 

We use the same notation as in \eqref{Eq:matrix} for  Systems $(S)$ and $(S^{\prime})$.
Our aim is to bound the number of solutions of $(S^{\prime})$ by $q^{mn-(2n-2h-2)}$.

We distinguish two cases.
\begin{enumerate}
    \item $s_{k}\geq 2$ for some $k \in \{1,\ldots,h+1\}$.

     Let $\mathcal{S}_k=\{E_1,E_2\}$.  
 The equations $(E_1)^{q^{\ell}},(E_2)^{q^\ell}$, $\ell=0,\ldots,n-k-1$ form a system of $2(n-k)$
 independent equations in the variables $x_{i,j}$ and thus the number of solutions of $(S^{\prime})$ is bounded by  $q^{mn-2(n-k)}$. Note that $mn-2(n-k)\leq mn-2(n-h-1)$.

 \item $s_1=s_2=\cdots=s_h=s_{h+1}=1$

Let $\mathcal{S}_1=\{E_1\}$. The equation $E$ on its own sets a maximum limit of $q^{n(m-1)}$ for the total number of solutions for the entire system. Let $\mathcal{S}_i=\{E_i\}$, $i=2,\ldots,h+1$. Suppose that there exists $j=2,\ldots,h-1$ such that the equation $E$ (of degree $q^{j-1}$), obtained eliminating from $E_{j}$ a variable (say $x_1$) using $E_1$,   is non-vanishing. Now, the equations $(E)^{q^\ell}$, $\ell=0,\ldots,n-j$, form a system of $(n-j+1)$ independent equations in $(m-1)n$ variables, $x_{i,j}$, $i=2,\ldots,m$, $j=0,\ldots,n-1$, and thus the number of solutions of $(S^{\prime})$ is at most $q^{n(m-1)-n+h-1}<q^{nm-2n+2h+2}$.


Suppose that there exists no such a $j$. Thus, System $(S^{\prime})$ reads as
$$\begin{cases} A^{(1)}_{0}\cdot(x_{1},\dots,x_{m})^T=0\\
A^{(h+1)}_0\cdot(x_1,\dots,x_m)^T+A^{(h+1)}_1\cdot(x_{1}^q,\dots,x_{m}^q)^T+\cdots+A^{(h+1)}_{h-1}\cdot(x_{1}^{q^{h-1}},\dots,x_{m}^{q^{h-1}})^T+\\+(c_1,\dots,c_m)\cdot(x_1\qih,\dots,x_m\qih)^T+(\alpha_1c_m,\dots,\alpha_mc_{m-1})\cdot(x_1\qjh,\dots,x_m\qjh)^T=0.
\end{cases},$$
for some $c_1,\ldots,c_m\in \mathbb{F}_{q^n}$.



 First, observe that $(c_1^q,c_2^q,\ldots,c_m^q)$ cannot be proportional to $(\alpha_1c_m,\alpha_2c_1,\ldots,\alpha_mc_{m-1})$.
In fact, consider $\ell \in \mathbb{F}^*_{q^n}$  such that 
$$\ell(c_1^q,c_2^q,\ldots,c_m^q)=(\alpha_1c_m,\alpha_2c_1,\ldots,\alpha_mc_{m-1})$$
and observe that $c_i$ being zero for any $i=1,\dots,m$ yields  $(c_1,\dots,c_m)=(0,\dots,0)$, a contradiction.

So
\begin{eqnarray*}
    c_1^{q^{m}}=\ell^{-q^{m-1}}\alpha_{1}^{q^{m-1}}c_{m}^{q^{
    m-1}}=\cdots=L^{-1}\alpha_1^{q^{m-1}}\alpha_{2}^{q^{m-2}}\cdots\alpha_{m}c_1={L}^{-1}K_{\textbf{A}}c_1,   
\end{eqnarray*}
where
$L=\ell^{(q^{m-1}+q^{m-2}+\cdots+q+1)}$. So we obtain
$$K_{\textbf{A}}=(c_1^{q-1}\ell)^{(q^{m-1}+q^{m-2}+\cdots+q+1)},$$ a contradiction.

If $\left(A^{(1)}_{0}\right)^{q^{h+1}}$ is not proportional to $(\alpha_1c_m,\dots,\alpha_mc_{m-1})$, we can read the system $(S^{\prime})$ as 
$$\begin{cases}
    \left(A^{(1)}_{0}\right)^{q^{h+1}}\cdot(x_{1}^{q^{h+1}},\dots,x_{m}^{q^{h+1}})^T=0\\
    B_0\cdot(x_1,\dots,x_m)^T+B_1\cdot(x_{1}^q,\dots,x_{m}^q)^T+\cdots+B_{h-1}\cdot(x_{1}^{q^{h-1}},\dots,x_{m}^{q^{h-1}})^T+\\+(c_1,\dots,c_m)\cdot(x_1\qih,\dots,x_m\qih)^T+(\alpha_1c_m,\dots,\alpha_mc_{m-1})\cdot(x_1\qjh,\dots,x_m\qjh)^T=0,
\end{cases}$$
and we can proceed as in the case (a), and obtain $q^{mn-2(n-h-1)}$ as upper bound for the number of solutions of $(S^{\prime})$.

If $ \left(A^{(1)}_{0}\right)^{q^{h+1}}$ is proportional to $(\alpha_1c_m,\dots,\alpha_mc_{m-1})$, then it cannot be proportional to $(c_1^q,c_2^q,\ldots,c_m^q)$, so we can read the system $(S^{\prime})$ as 
$$\begin{cases}
\left(A^{(1)}_{0}\right)^{q^{h+1}}\cdot(x_{1}^{q^{h+1}},\dots,x_{m}^{q^{h+1}})^T=0\\
    B_0^{q}\cdot(x_1^{q},\dots,x_m^{q})^T+B_1^{q}\cdot(x_{1}^{q^2},\dots,x_{m}^{q^2})^T+\cdots+B_{h-1}^{q}\cdot(x_{1}^{q^{h}},\dots,x_{m}^{q^{h}})^T+\\+(c_1^{q},\dots,c_m^{q})\cdot(x_1\qjh,\dots,x_m\qjh)^T=0,
\end{cases}$$
and we can proceed as in the case (a), and obtain again $q^{mn-2(n-h-1)}$ as upper bound for the number of solutions of $(S^{\prime})$.

\end{enumerate}
\end{proof}

\noindent The previous theorem can be improved under suitable assumptions on $\textbf{A}$.

Indeed, let   $$\Pi_{i}:=\alpha_{i}^{q^{m-1}}\alpha_{i-1}^{q^{m-2}} \cdots \alpha_{i+2}^{q}  \alpha_{i+1},$$
where the indices of $\alpha_i$ are modulo $m$ in the range $[1,\dots,m]$, and let   $$\mathcal{B}:=\left\{\textbf{A}\in\mathcal{A}\mid\frac{\Pi_{\delta+2}}{\Pi_{2}} \textrm{ is not a $(q^{m}-1)$-power in $\field$  for any } \delta=1,\dots,m-1\right\}.$$

Note that such a set $\mathcal{B}$ is not empty when $m|n$ (cf. \cite[Prop. 2.13]{BGM2024}).

\begin{theorem}\label{genweis2}
     The subspace $V_{\mathbf{A},h}$ is $(m(h+1)-s(h+1),mn-(s+1)(n-h-1))$-evasive for every $s=2,\dots,m-2$, $\textbf{A}\in\mathcal{B}$, $h\geq2$.
\end{theorem}
\begin{proof}
   Consider an $\mathbb{F}_{q^n}$-subspace $H$ of $V(m(h+1),q^{n})$ of dimension $(m-s)(h+1)$ and thus defined by $s(h+1)$ independent linear homogeneous equation in $z_1,\ldots,z_{m(h+1)}$. 
    
    

We use the same notation as in \eqref{Eq:matrix} for  Systems $(S)$ and $(S^{\prime})$.

Our aim is to bound the number of solutions of $(S^{\prime})$ by $q^{mn-(s+1)(n-h-1)}$.

We distinguish two cases.
\begin{enumerate}
    \item $s_{k}\geq s+1$ for some $k \in \{1,\ldots,h+1\}$.

     Let $\mathcal{S}_k=\{E_1,E_2,\dots,E_{s+1}\}$.  
 The equations $(E_1)^{q^{\ell}},(E_2)^{q^\ell},\dots,(E_{s+1})^{q^\ell}$, $\ell=0,\ldots,n-k-1$, form a system of $(s+1)(n-k)$
 independent equations in the variables $x_{i,j}$ and thus the number of solutions is bounded by  $q^{mn-(s+1)(n-k)}$. Note that $mn-(s+1)(n-k)\leq mn-(s+1)(n-h-1)$.

 \item $s_1=s_2=\cdots=s_h=s_{h+1}=s$.
 Let $\mathcal{S}_1=\{E_1,\ldots,E_{s}\}$. These equations  on their own set a maximum  of $q^{n(m-s)}$ for the total number of solutions for the entire system.  Suppose that there exists $j=2,\ldots,h-1$ such that the equations in $\mathcal{S}_j$ (of degree $q^{j-1}$), once combined $\mathcal{S}_1$, are not all vanishing. Let $E$ be one of these equations: it depends on $(m-s)$ variables $x_{i_1},\ldots,x_{i_{m-s}}$. Now,  $(E)^{q^\ell}$, $\ell=0,\ldots,n-j$, form a system of $(n-j+1)$ independent equations in $(m-s)n$ variables, $x_{i_\ell,j}$, $\ell=1,\ldots,m-s$, $j=0,\ldots,n-1$, and thus the number of solutions of $(S^{\prime})$ is at most $q^{n(m-s)-n+h-1}<q^{nm-(s+1)n+(s+1)(h+1)}$.


Suppose that there exists no such a $j$. Thus, System $(S^{\prime})$ reads as
$$\begin{cases} \left(A^{(1)}_{0}\right)^{q^{h+1}} (x_{1}^{q^{h+1}},\dots,x_{m}^{q^{h+1}})^T=0,\\
A^{(h+1)}_0 (x_1,\dots,x_m)^T+A^{(h+1)}_1 (x_{1}^q,\dots,x_{m}^q)^T+\cdots+A^{(h+1)}_{h-1} (x_{1}^{q^{h-1}},\dots,x_{m}^{q^{h-1}})^T+\\+A^{(h+1)}_{h} (x_1\qih,\dots,x_m\qih)^T+\widetilde{A}^{(h+1)}_{h}(x_1\qjh,\dots,x_m\qjh)^T=0,
\end{cases}$$
where 
$$A^{(h+1)}_{h}=\begin{pmatrix}
H_1&H_2&\ldots&H_m
\end{pmatrix},  \quad 
\widetilde{A}^{(h+1)}_{h}:= \begin{pmatrix}
\alpha_1 H_m&\alpha_2 H_1&\ldots&\alpha_mH_{m-1}\end{pmatrix}.$$



First, we prove by contradiction that the rowspan of $\left(A^{(h+1)}_{h}\right)^q$ is not equal to the rowspan of $\widetilde{A}^{(h+1)}_{h}$.

Let $L \in \mathrm{GL}(s,q^n)$ be such that 
$$L\left(A^{(h+1)}_{h}\right)^q=\widetilde{A}^{(h+1)}_{h},$$
that is 
$$L H_1^{q}=\alpha_1 H_m, \; L H_2^{q}=\alpha_2 H_1, \ldots, \; L H_{m-1}^{q}=\alpha_{m-1} H_{m-2},\; L H_m^{q}=\alpha_m H_{m-1},$$
i.e.
$$H_1^{q}=\alpha_1 L^{-1}H_m, \; H_2^{q}=\alpha_2 L^{-1} H_1, \ldots, \; H_{m-1}^{q}=\alpha_{m-1} L^{-1} H_{m-2},\; H_m^{q}=\alpha_mL^{-1} H_{m-1}.$$     

For any $i\in \{1,\ldots,m\}$
\begin{eqnarray*}
H_i^{q^{m}}&=&\alpha_{i}^{q^{m-1}} (L^{-1})^{q^{m-1}}H_{i-1}^{q^{m-1}}\\&=&\alpha_{i}^{q^{m-1}} (L^{-1})^{q^{m-1}}\alpha_{i-1}^{q^{m-2}} (L^{-1})^{q^{m-2}}H_{i-2}^{q^{m-2}}\\
&=&\alpha_{i}^{q^{m-1}}\alpha_{i-1}^{q^{(m-2)K}} \cdots \alpha_{i+2}^{q} (L^{-1})^{q^{m-1}} (L^{-1})^{q^{m-2}} \cdots (L^{-1})^{q} H_{i+1}^{q}\\
&=&\alpha_{i}^{q^{m-1}}\alpha_{i-1}^{q^{m-2}} \cdots \alpha_{i+2}^{q}  \alpha_{i+1} (L^{-1})^{q^{m-1}} (L^{-1})^{q^{m-2}} \cdots (L^{-1})^{q}(L^{-1}) H_{i}.
\end{eqnarray*}
Write the above equation as 
$$H_i^{q^{m}}=\Pi_{i} \overline{L} H_i,$$
where $\overline{L}=(L^{-1})^{q^{m-1}} (L^{-1})^{q^{m-2}} \cdots (L^{-1})^{q}(L^{-1})$.

Since the rank of $A^{(h+1)}_{h}$ is $s$ and $s<m$ there exist 
linearly independent columns $H_{i_1},\ldots, H_{i_s}$ such that $H_{i_{s+1}}$ can be written as
$$H_{i_{s+1}}=\lambda_1 H_{i_1}+\cdots+ \lambda_sH_{i_s},$$
for some $\lambda_i \in \mathbb{F}_{q^n}$ not all vanishing (since $H_{i_{s+1}}$ being the zero column would yield $A^{(h+1)}_{h}=0$, a contradiction).

Thus, 

\begin{eqnarray*}
(\lambda_1 H_{i_1}+\cdots+ \lambda_sH_{i_s})^{q^{m}}&=&\Pi_{i_{s+1}} \overline{L} (H_{i_{s+1}})=\Pi_{i_{s+1}} \overline{L} (\lambda_1 H_{i_1}+\cdots+ \lambda_sH_{i_s})\\
&=&\Pi_{i_{s+1}} (\lambda_1 H_{i_1}^{q^{m}}/\Pi_{i_1}+\cdots+ \lambda_s  H_{i_s}^{q^{m}}/\Pi_{i_s}).
\end{eqnarray*}
Since $H_{i_1}^{q^{m}},\ldots, H_{i_s}^{q^{m}}$ are also independent,
$$\lambda_j^{q^{m}}=\frac{\Pi_{i_{s+1}}}{\Pi_{i_j}}\lambda_j, \qquad \forall j=1,\ldots,s.$$

This contradicts that $\frac{\Pi_{\delta+2}}{\Pi_{2}}$ is not a $(q^{m}-1)$-power for any $\delta= {1,\ldots,m}$, since $\Pi_{\delta}^{q^{\gamma }}=(\xi_{\delta,\gamma})^{q^{m}-1}\Pi_{\delta+\gamma}$ for some $\xi_{\delta,\gamma}\in\field$, so $\text{rowspan}(A^{(h+1)}_{h})\neq\text{rowspan}(\widetilde{A}^{(h+1)}_{h})$.

We distinguish now two subcases:
\begin{enumerate}
    \item $\text{rowspan}(\widetilde{A}^{(h+1)}_{h})\neq\text{rowspan}\left(A^{(1)}_{0}\right)^{q^{h+1}} $\\
    It is readily seen that $(S^{\prime})$ contains at least $s+1$ independent equations of degree $h+1$, and thus, arguing as above, we deduce that System $(S^\prime)$ as at most $q^{mn-(s+1)(n-h-1)}$ solutions.
\item $\text{rowspan}(\widetilde{A}^{(h+1)}_{h})=\text{rowspan}\left(A^{(1)}_{0}\right)^{q^{h+1}} $\\
Then $(S^{\prime})$ is equivalent to  
$$\begin{cases} \left(A^{(1)}_{0}\right)^{q^{h+1}} (x_{1}^{q^{h+1}},\dots,x_{m}^{q^{h+1}})^T=0,\\
\left(A^{(h+1)}_0\right)^q (x_1^q,\dots,x_m^q)^T+\left(A^{(h+1)}_1\right)^q(x_{1}^{q^2},\dots,x_{m}^{q^2})^T+\cdots+\\+\left(A^{(h+1)}_{h-1}\right)^q (x_{1}^{q^{h}},\dots,x_{m}^{q^{h}})^T+\left({A}^{(h+1)}_{h}\right)^q(x_1\qjh,\dots,x_m\qjh)^T=0.
\end{cases}$$
Since $\text{rowspan}({A}^{(h+1)}_{h})\neq\text{rowspan}\left(A^{(1)}_{0}\right)^{q^{h+1}},$ we can proceed as before and obtain the desired bound on the number of solutions of $(S^{\prime})$.
\end{enumerate}

\end{enumerate}

\end{proof}
In the sequel, we will denote by $\mathcal{C}$ and $\mathcal{G}_{\mathbf{A},h}$ an MRD code arising from an $h$-scattered $\F_q$-subspace $U$ of $V(m(h+1),q^n)$ obtained as direct sum of $m$ maximum $h$-scattered $\F_q$-subspaces of $V((h+1),q^n)$ and the MRD code associated with the set $V_{\mathbf{A},h}$, respectively.

\begin{theorem}
    The following hold.
    \begin{enumerate}
    \item If $n> (s+1)(h+1)$, then for any $s=2,\dots,m-2$ and for any $\textbf{A}\in\mathcal{B}$
    $$d_{s(h+1)}(\mathcal{G}_{\mathbf{A},h})>d_{s(h+1)}(\mathcal{C}).$$
    \item If $n>2h+2$, then for any $\textbf{A}\in\mathcal{A}$
    $$d_{h+1}(\mathcal{G}_{\mathbf{A},h})> d_{h+1}(\mathcal{C}).$$
    \item  For any $\textbf{A}\in\mathcal{A}$ $$d_{(m-1)(h+1)}(\mathcal{G}_{\mathbf{A},h}) > d_{(m-1)(h+1)}(\mathcal{C}).$$
    \end{enumerate}
\end{theorem}
\begin{proof}
\begin{enumerate}
    \item  From Theorem \ref{genweis2} we have $d_{s(h+1)}(\mathcal{G}_{\mathbf{A},h})\geq (s+1)(n-h-1)$. In addition, Proposition \ref{prop:genweights_directsum} yields $d_{s(h+1)}(\mathcal{C})\leq sn$. The claim follows.
    \item  From Theorem \ref{genweis1} we have $d_{h+1}(\mathcal{G}_{\mathbf{A},h})\geq 2n-2h-2$. Proposition \ref{prop:genweights_directsum} yields $d_{h+1}(\mathcal{C})\leq n$ and thus the claim.
    \item From Theorem \ref{thm:evas} we have $d_{(m-1)(h+1)}(\mathcal{G}_{\mathbf{A},h})\geq mn-h-2$. By  Proposition \ref{prop:genweights_directsum},  $d_{(m-1)(h+1)}(\mathcal{C})\leq (m-1)n$, and thus the claim follows.
\end{enumerate}
\end{proof}

\begin{remark}Arguing as in the proof of \cite[Theorem 2.11]{BGM2024}, we can note that each subspace $V_{\mathbf{A},h}$ is  $h$-scattered in infinitely many extensions of $\F_{q^n}$.
\end{remark}

\section*{Acknowledgements}
The authors thank the Italian National Group for Algebraic and Geometric Structures and their Applications (GNSAGA—INdAM)
which supported the research. 

\section*{Declarations}
{\bf Conflicts of interest.} The authors have no conflicts of interest to declare that are relevant to the content of this
article.

\bibliographystyle{abbrv}

\begin{thebibliography}{10}

\bibitem{BBL2000}
S.~Ball, A.~Blokhuis, and M.~Lavrauw.
\newblock Linear {$(q+1)$}-fold blocking sets in {${\rm PG}(2,q^4)$}.
\newblock {\em Finite Fields Appl.}, 6(4):294--301, 2000.

\bibitem{BaCsMT2020}
D.~Bartoli, B.~Csajb{\'o}k, G.~Marino, and R.~Trombetti.
\newblock Evasive subspaces.
\newblock {\em Journal of Combinatorial Designs}, 29(8):533--551, 2021.

\bibitem{BGGM}
D.~Bartoli, A.~Giannoni, F.~Ghiandoni, and G.~Marino.
\newblock A new family of 2-scattered subspaces and related {MRD} codes.
\newblock {\em arXiv preprint arXiv:2403.01506}, 2024.

\bibitem{BGM2024}
D.~Bartoli, A.~Giannoni, and G.~Marino.
\newblock New scattered subspaces in higher dimensions.
\newblock {\em arXiv preprint arXiv:2402.15223}, 2024.

\bibitem{BGMP2015}
D.~Bartoli, M.~Giulietti, G.~Marino, and O.~Polverino.
\newblock Maximum scattered linear sets and complete caps in {G}alois spaces.
\newblock {\em Combinatorica}, 38(2):255--278, 2018.

\bibitem{BMNV2022}
D.~Bartoli, G.~Marino, A.~Neri, and L.~Vicino.
\newblock Exceptional scattered sequences.
\newblock {\em arXiv preprint arXiv:2211.11477}, 2022.

\bibitem{BZ2018}
D.~Bartoli and Y.~Zhou.
\newblock Exceptional scattered polynomials.
\newblock {\em J. Algebra}, 509:507--534, 2018.

\bibitem{Berger}
T.~Berger.
\newblock Isometries for rank distance and permutation group of gabidulin
  codes.
\newblock {\em IEEE Trans. Inf. Theory}, 49 (11), 2003.

\bibitem{BL2000}
A.~Blokhuis and M.~Lavrauw.
\newblock Scattered spaces with respect to a spread in {PG}$(n, q)$.
\newblock {\em Geom. Dedicata}, 81(1):231--243, 2000.

\bibitem{CSMPZ2016}
B.~Csajb{\'o}k, G.~Marino, O.~Polverino, and F.~Zullo.
\newblock Maximum scattered linear sets and {MRD}-codes.
\newblock {\em J. Algebraic Combin.}, 46(3):517--531, 2017.

\bibitem{CsMPZ2019}
B.~Csajb{\'o}k, G.~Marino, O.~Polverino, and F.~Zullo.
\newblock Generalising the scattered property of subspaces.
\newblock {\em Combinatorica}, 41(2):237--262, 2021.

\bibitem{delsarte1978bilinear}
P.~Delsarte.
\newblock Bilinear forms over a finite field, with applications to coding
  theory.
\newblock {\em J. Combin. Theory Ser. A}, 25(3):226--241, 1978.

\bibitem{ducoat2015generalized}
J.~Ducoat and G.~Kyureghyan.
\newblock Generalized rank weights: a duality statement.
\newblock {\em Topics in Finite Fields}, 632:101--109, 2015.

\bibitem{gabidulin1985theory}
E.~M. Gabidulin.
\newblock Theory of codes with maximum rank distance.
\newblock {\em Problemy Peredachi Informatsii}, 21(1):3--16, 1985.

\bibitem{kurihara2015relative}
J.~Kurihara, R.~Matsumoto, and T.~Uyematsu.
\newblock Relative generalized rank weight of linear codes and its applications
  to network coding.
\newblock {\em IEEE Transactions On information theory}, 61(7):3912--3936,
  2015.

\bibitem{marino2023evasive}
G.~Marino, A.~Neri, and R.~Trombetti.
\newblock Evasive subspaces, generalized rank weights and near mrd codes.
\newblock {\em Discrete Mathematics}, 346(12):113605, 2023.

\bibitem{Randrianarisoa2020geometric}
T.~H. Randrianarisoa.
\newblock A geometric approach to rank metric codes and a classification of
  constant weight codes.
\newblock {\em Des. Codes Cryptogr.}, 88(7):1331--1348, 2020.

\bibitem{sheekey_new_2016}
J.~Sheekey.
\newblock A new family of linear maximum rank distance codes.
\newblock {\em Adv. Math. Commun.}, 10(3):475, 2016.

\bibitem{ShVdV}
J.~Sheekey and G.~V. de~Voorde.
\newblock Rank-metric codes, linear sets, and their duality.
\newblock {\em Designs, Codes and Cryptography}, 88(4):655--675, dec 2019.

\bibitem{zini2021scattered}
G.~Zini and F.~Zullo.
\newblock Scattered subspaces and related codes.
\newblock {\em Des. Codes Cryptogr.}, 89(8):1853--1873, 2021.

\end{thebibliography}

\Addresses
\end{document}